\documentclass[11pt,oneside]{amsart}

\usepackage[utf8]{inputenc}
\usepackage[english]{babel}
\usepackage{csquotes}
\usepackage[T1]{fontenc}
\usepackage[final]{microtype}

\usepackage{amsmath}
\usepackage{amsfonts,amssymb,amsthm}
\usepackage{xfrac}
\newcommand{\F}{\mathbb{F}}
\newcommand{\Z}{\mathbb{Z}}
\theoremstyle{plain}
\newtheorem{theorem}{Theorem}[section]
\newtheorem{proposition}[theorem]{Proposition}
\theoremstyle{definition}
\newtheorem{definition}[theorem]{Definition}

\usepackage[backend=biber,style=alphabetic]{biblatex}
\renewbibmacro{in:}{}
\usepackage{geometry}
\usepackage[bookmarks]{hyperref}
\usepackage[capitalize]{cleveref}

\title{A Small Maximal Sidon Set in $\Z_2^n$}

\author{Maximus Redman
\and Lauren Rose
\and Raphael Walker}
\date{}

\begin{document}

\maketitle

\begin{abstract}
    A Sidon set is a subset of an Abelian group with the property that each sum of two distinct elements is distinct. We construct a small maximal Sidon set of size $O((n \cdot 2^n)^{\sfrac{1}{3}})$ in the group $\Z_2^n$, generalizing a result of Ruzsa concerning maximal Sidon sets in the integers.
\end{abstract}

\section{Introduction}

\noindent Sidon sets were first studied by Simon Sidon in the 1930s \cite{sidon}, and then by Paul Erd\H{o}s in the 1940s \cite{erdos}. At this time, they were sometimes referred to as Sidon sequences, given that they were assumed to be subsets of the integers. However, the notion of Sidon set generalizes quite naturally to other Abelian groups, and may be generalized to arbitrary groups in a number of ways. We limit our discussion here to Sidon sets contained in Abelian groups, with the following definition.

\begin{definition}
    Let $G$ be an Abelian group and $S \subseteq G$. We say $S$ is \textbf{Sidon} if whenever $a+b = x+y$ and $a, b, x, y$ are in $S$, we have $\{a, b\} = \{x, y\}$.
\end{definition}

By such a definition, if every element of $G$ has order 2, then every Sidon set in $G$ consists of a single element. Thus in the case $G = \Z_2^n$, we further require that $a \neq b$ and that $x \neq y$, i.e. that the sum of every pair of \emph{distinct} elements be distinct.

The case $G = \Z_2^n$ is the subject of this paper. For these groups, the initial question of interest is the largest possible size of a Sidon set, which is known to be $\Theta(2^{\sfrac{n}{2}})$ (for example, \cite{lindstrom, bch, taitwon2021}). However, another question of interest, which has been much less studied, is the \emph{minimal} size of a maximal Sidon set, defined below.

\begin{definition}
    We say a Sidon set $S \subseteq G$ is \textbf{maximal} if there exists no Sidon set $S' \subseteq G$ with $S \subset S'$.
\end{definition}

The union of a maximal Sidon set with any additional element is not a Sidon set. That is, for any $x$ not in $S$, the equation $x+a = b+c$ is solvable for $a, b, c \in S$. The main result of this paper is the construction of maximal Sidon set in $\Z_2^n$ of size $O((n \, 2^n)^{\sfrac{1}{3}})$, in \cref{thm:small-maximal}.

\section{Large Sidon Sets in \texorpdfstring{$\Z_2^n$}{\textit{(Z/2Z)\^{}n}}}\label{sec:thick}

\noindent The smallest known construction of a maximal Sidon set in the integers is due to Ruzsa \cite{ruzsa}. Ruzsa's method for constructing a small maximal Sidon set relies on the existence of a sufficiently dense Sidon set in a smaller space than the one of ultimate interest: for Sidon sets in the integers, Singer's theorem \cite{singer} on the existence of a Sidon set of size $p+1$ in $\Z_{1 + p + p^2}$ sufficed to provide this smaller set. 

A construction in the case of $\Z_2^n$ comes from the field of coding theory, in particular the Bose–Chaudhuri–Hocquenghem (BCH) codes \cite{bch}, which are a type of error-correction system making use of polynomials in finite fields. When constructed over $\F_2$, BCH codes of length $n$ and distance 5 are in fact Sidon sets. The Sidon sets resulting from such constructions are suitable for the proof of \cref{thm:small-maximal}, our generalization of Ruzsa's construction.

\begin{proposition}\label{prop:tait-won}
    Let $S_n \subset \F_{2^{n/2}} \times \F_{2^{n/2}}$ be defined by
    \[
        \{(x, x^3) \mid x \in \F_{2^{n/2}}\}.
    \]
    Then $S_{n}$ is a Sidon set for all even $n \geq 1$.
\end{proposition}

\begin{proof}
    Apply Lemma~2 in \cite[74]{bch} with $m=n/2$, $l=4$, and $t=2$.
\end{proof}

As $\F_{2^{n/2}} \times \F_{2^{n/2}}$ is additively isomorphic to $\Z_2^n$, the image of $S_n$ under any such isomorphism is a Sidon set in $\Z_2^n$. By abuse of notation, we will henceforth consider $S_n \subset \Z_2^n$ to be some such image.

If $S$ is a Sidon set in $G$, for some points $x$ of $G \setminus S$, $S \cup \{x\}$ is a Sidon set, and for other points it is not: that is, we have $x + a = b + c$, for some $a, b, c \in S$. Recall that for a maximal Sidon set,  $S \cup \{x\}$ is never a Sidon set. Yet the number of ``collisions'' --- that is, the number of solutions to $x+a=b+c$ --- varies depending on the choice of $x$. We use the following definition to measure the number of such solutions.

\begin{definition}\label{def:cover}
    Let $S \subseteq \Z^n_2$ be a Sidon set. We say a point $x \in \Z_2^n \setminus S$ is \textbf{covered $k$ times by $S$} if there exist $k$ distinct unordered solutions $\{a,b,c\}$ to 
    \[
        a + b + c = x,
    \]
    for $a, b, c \in S$.
\end{definition}

It follows immediately that if a point $x$ is covered $k$ times by $S$, then the $k$ triples $\{a, b, c\}$ satisfying $a + b + c = x$ are disjoint by the Sidon property of $S$. The maximal Sidon sets are exactly those that cover every point at least once, and maximal Sidon sets that cover each point many times are relatively ``denser'' than those which cover each point fewer times in the same space. Thus to search for large Sidon sets, we try to maximize how many times each point is covered, while to search for small Sidon sets, we want to minimize the same quantity.

We show that the BCH Code construction of $S_{2n} \subset \Z_2^{2n}$ (above) covers every point of $\Z_2^{2n} \setminus S_{2n}$ relatively many times. Note that notationally, we say $f(n) = \Omega(g(n))$ if there exists $c > 0$ and $n_0$ such that for all $n > n_0$, $f(n) \geq c g(n)$. Compare to the more common big-O notation: $f(n) = O(g(n))$ if $f(n) \leq c g(n)$.

\begin{theorem}\label{thm:k-cover}
    The set $S_{2n}$ given by \cref{prop:tait-won} covers every element of ${\Z_2^{2n} \setminus S_{2n}}$ at least $\Omega(2^n)$ times.
\end{theorem}
\begin{proof}

    As in \cref{prop:tait-won}, let
    \[
        S = \{(x, x^3) \mid x \in \F_{2^n}\}.
    \]
    Let $(x, y) \in \F_{2^n}^2 \setminus S$. That is, let $x, y \in \F_{2^n}$, where $y \neq x^3$. To determine how many times $S$ covers $(x, y)$, we count the number of solutions $(a,b,c)$ over $\F_{2^n}$ to
    \begin{align}
        a + b + c &= x  \label{eqn:cover1}\\
        a^3 + b^3 + c^3 &= y \label{eqn:cover2}.
    \end{align}
    
    If \cref{eqn:cover1} holds, then $c = x+a+b$, and it suffices to solve the equation
    \begin{equation*}
        0 = y + a^3 + b^3 + (x + a + b)^3.
    \end{equation*}
    We can homogenize this to get a projective curve $C$ of degree 3, defined by
    \[
        F(a, b, p) = a^2b + b^2a + a^2px + b^2px + ap^2x^2 + bp^2x^2 + p^3x^3 + p^3y = 0.
    \]
    The polynomial $F$ is absolutely irreducible: that is, irreducible over $\overline{\F}_{2^n}[a, b, p]$, where $\overline{\F}_{2^n}$ is the algebraic closure of $\F_{2^n}$. To see this, specialize to a polynomial in $\overline{\F}_{2^n}[a, b]$ by setting $p = 1$. Supposing that $F(a, b, 1)$ is the product of two non-constant polynomials, we have without loss of generality
    \[
        F(a, b, 1) = (s_1 a^2 + s_2 b^2 + s_3 ab + s_4 a + s_5 b + s_6)(t_1 a + t_2 b + t_3).
    \]
    Expanding the product and setting the coefficient of each resulting term equal to the coefficient of the corresponding term in $F(a, b, 1)$ gives the following system of equations, which we will solve by hand:
    \begin{align}
        s_1 t_1 &= 0 \label{eqn:a3}\\ 
        s_2 t_2 &= 0 \label{eqn:b3}\\ 
        s_3 t_1 + s_1 t_2 &= 1 \label{eqn:a2b}\\ 
        s_2 t_1 + s_3 t_2 &= 1 \label{eqn:ab2}\\ 
        s_4 t_1 + s_1 t_3 &= x \label{eqn:a2}\\ 
        s_5 t_2 + s_2 t_3 &= x \label{eqn:b2}\\ 
        s_5 t_1 + s_4 t_2 + s_3 t_3 &= 0 \label{eqn:ab}\\ 
        s_6 t_1 + s_4 t_3 &= x^2 \label{eqn:a}\\ 
        s_6 t_2 + s_5 t_3 &= x^2 \label{eqn:b}\\ 
        s_6 t_3 &= x^3 + y. \label{eqn:const}
    \end{align}

    By \eqref{eqn:a3}, the possibilities for $(s_1, t_1)$ are $(1, 0)$, $(0, 1)$, or $(0, 0)$ without loss of generality. \eqref{eqn:a2b} rules out the possibility that $s_1 = t_1 = 0$, so it remains to check the two cases $s_1 = 0, t_1 = 1$ and $s_1 = 1, t_1 = 0$. 

    Suppose $s_1 = 1, t_1 = 0$. By \eqref{eqn:a2b} we have $t_2 = 1$ and hence $s_2 = 0$. By \eqref{eqn:ab2} we have $s_3 = 1$; by \eqref{eqn:a2}, $t_3 = x$; by \eqref{eqn:b2}, $s_5 = x$; by \eqref{eqn:ab}, $s_4 = x$; by \eqref{eqn:b}, $s_6 = 0$; and hence by \eqref{eqn:const} that $x^3 = y$, which contradicts the assumption that $(x, y) \not\in S$.

    On the other hand, suppose $s_1 = 0, t_1 = 1$. Then we have by \eqref{eqn:a2b} that $s_3 = 1$ and by \eqref{eqn:ab2} that $s_2 + t_2 = 1$. Combining this with \eqref{eqn:b3}, we have that either $s_2 = 1$ and $t_2 = 0$ or $s_2 = 0$ and $t_2 = 1$. Since the polynomial $F(a, b, 1)$ is symmetric in $a$ and $b$, (i.e, $F(a, b, 1) = F(b, a, 1)$), the case $s_2 = 1$ and $t_2 = 0$ reduces to the case above, which we have seen to be inconsistent. So without loss of generality we have $s_2 = 0, t_2 = 1$. Then by \eqref{eqn:a2} we have $s_4 = x$; by \eqref{eqn:b2}, $s_5 = x$; by \eqref{eqn:ab}, $t_3 = 0$; and hence by \eqref{eqn:const} that $x^3 = y$ which is as before a contradiction.

    Thus the system of equations for the factorization of $F(a, b, 1)$ into two non-constant polynomials is inconsistent when $x^3 \neq y$, and so $F(a, b, 1)$ is absolutely irreducible. As the specialization in $\overline{\F}_{2^n}[a, b]$ is irreducible, so is the original polynomial in $\overline{\F}_{2^n}[a, b, p]$.
    
    Since $F$ is absolutely irreducible, by the Hasse-Weil theorem (see, for instance, \cite[6]{hurt}),
    \[
        2^n + 1 - 2g\sqrt{2^n} \leq \#C \leq 2^n + 1 + 2g\sqrt{2^n},
    \]
    where $g$ is the geometric genus of $C$ and $\#C$ is the number of projective points on $C$ (i.e., the number of solutions up to scalar multiplication of $(a, b, p)$).

    By the Riemann-Hurwitz formula (see, for instance, \cite[37]{silverman}), the genus of $C$ can be computed from the degree $d$ and the number of singular points $s$ of $C$. In particular,
    \[
        g = \frac{(d-1)(d-2)}{2} - s = 1 - s \leq 1,
    \]
    and hence
    \[
        2^n + 1 - 2\sqrt{2^n}\leq \#C \leq 2^n + 1 + 2\sqrt{2^n}.
    \]
    
    The solutions of interest are the affine solutions, given by
    \[
        F(a, b, 1) = 0.
    \]
    
    Let $A$ be the number of affine solutions as above, and $P$ the number of projective solutions, given by $F(a, b, 0) = 0$. We have $\#C = A + P$.
    
    $P$ is the number of solutions to 
    \[
        F(a, b, 0) = a^2 b + b^2 a = a b (a + b) = 0,
    \]
    which is satisfied by the pairs $(0, 1), (1, 0)$, and $(1, 1)$, and so $P = 3$. Then we have
    \[
        A = \#C - P \geq 2^n - 2\sqrt{2^n} - 2.
    \]
    
    Each of the $A$ affine points $(a, b, 1)$ on $C$ corresponds to an ordered triple of distinct elements $(a, b, a+b+x)$ such that $(a, a^3) + (b, b^3) + (a+b+x, (a+b+x)^3) = (x, y)$. From our assumption that $(x, y) \not\in S$ (i.e., $y \neq x^3$), it follows that for each pair, $a$, $b$, and $a+b+x$ are distinct. Thus the point $(x, y)$ is covered $A/6 \geq \frac{2^n - 2\sqrt{2^n} - 2}{6} = \Omega(2^n)$ times by $S$.
\end{proof}

For $n = 3, 5, 7, 9$, we have observed by direct computation that the Sidon set given by the method above covers each point \emph{exactly} $\frac{2^n - 2}{6}$ times, and we conjecture that the pattern continues for larger odd $n$. Nonetheless, the fact that each point is covered $\Omega(2^n)$ times is sufficiently strong for the requirements of the main theorem of this paper, in which we construct a small maximal Sidon set in $\Z_2^n$ by ``projecting'' this Sidon set $S$ from a subspace.

\section{An Analogue of Ruzsa's Construction}\label{sec:thin}

In this section we discuss maximal Sidon sets and their sizes. While rather tight bounds on the largest possible size of a Sidon set are known for many groups (the size of the largest Sidon set in $\Z_q$ is around $\sqrt{q}$; see, for example, \cite{obryant}), very little is known about the minimal sizes of maximal Sidon sets. 

In the case of $G = \Z_2^n$, any maximal Sidon set $S$ satisfies
\[
    \binom{|S|}{3} + |S| \geq 2^n,
\]
as $S$ is maximal if and only if every point of $G \setminus S$ is covered at least once by $S$, and hence $|S| = \Omega(2^{\sfrac{n}{3}})$. We suspect that this bound is not sharp in general.
On the other hand, there has not been until now an \emph{upper} bound on the size of the smallest maximal Sidon set in $\Z_2^n$ except for the upper bound on the size of \emph{any} Sidon set: $|S| = O(2^{\sfrac{n}{2}})$.
In the case of the integers, Ruzsa \cite{ruzsa} showed that there exists a maximal Sidon set in $[1, N]$ with size $O((N \log N)^{\sfrac{1}{3}})$.

The main theorem of this paper is a similar bound on the size of maximal Sidon sets in $\Z_2^n$.

\begin{theorem}\label{thm:small-maximal}
    There exists a maximal Sidon set $S \subseteq \Z_2^n$ such that
    \[
        |S| = O\left((n \cdot 2^n)^{\sfrac{1}{3}}\right).
    \]
\end{theorem}

In this section, we adopt the technique from \cite{ruzsa} in order to construct a small maximal Sidon set in $\Z_2^n$, using a method that generalizes easily to arbitrary Abelian groups, provided a sufficiently dense Sidon set can be found in a quotient of the desired group. In the future, we hope to use this method to construct small maximal Sidon sets in groups of the form $\Z_p^n$, for $p > 2$.

\begin{proof}[Proof of \cref{thm:small-maximal}]
    Let $T > 0$ be a fixed constant, such that for each $t$, the Sidon set $S_{2t} \subset \Z_2^{2t}$ given by \cref{thm:k-cover} covers every point of $\Z_2^{2t} \setminus S_{2t}$ at least $\frac{2^t}{T}$ times (i.e., $\Omega(2^t)$).

Let $m$ be the least even integer satisfying
\[
    m > \frac{2}{3} \log_2(T \ln(2) \, n \, 2^n),
\]
and let $Q < \Z_2^n$ be a subgroup isomorphic to $\Z_2^{n-m}$. Observe that the quotient group $\Z_2^n/Q$ is isomorphic to $\Z_2^m$, which is the additive group of $\F_2^m$.
For an element $x$ of $\Z_2^n$, let $\overline x$ be the coset $x + Q$.

By \cref{thm:k-cover}, there exists a Sidon set $A \subset \Z_2^n/Q$ (i.e., $A = S_m$) such that $|A| = 2^{\sfrac{m}{2}}$ and $A$ covers every $p \in \Z_2^n/Q \setminus A$ at minimum $\frac{2^{\sfrac{m}{2}}}{T}$ times.

For each $a_i$ in $A$, pick a random representative $b_i$ of the coset $a_i$, and let $B = \{b_i \mid 1 \leq i \leq 2^{\sfrac{m}{2}}\}$. Choose each representative $b_i$ from a uniform distribution on $a_i$, independently of the random choices for each other representative, such that each of the $(2^{n-m})^\frac{m}{2}$ possible choices for $B$ has the same probability. Regardless of which representatives are chosen, $B$ is a Sidon set in $\Z_2^n$.

Any such set $B$ can be extended to a maximal Sidon set $S$. We can add an element $x$ to $B$ and still have a Sidon set if any only if $x$ is not covered by $B$: that is, if there is no solution to
\begin{equation*}
    x = a + b + c
\end{equation*}
for $a, b, c \in B$.

We show that it is possible to choose $B$ such that every $x$ satisfying $\overline x \not\in A$ is covered by $B$. That is, if $A$ covers $\overline x$, then $B$ covers $x$.
Since we chose each $b_i$ randomly, we determine the probability that such an element (i.e., with $\overline x \not\in A$) is covered by $B$.

Since $A$ covers every element of $\Z_2^n/Q \setminus A$ at least $2^{m/2}/T$ times, let $(a_{u_j}, a_{v_j}, a_{w_j})$ for $1 \leq j \leq J$ (where $J \geq 2^{m/2}/T$) be a sequence of disjoint triples of elements of $A$, such that for each $j$,
\[
    a_{u_j} + a_{v_j} + a_{w_j} = \overline x,
\]
and hence $b_{u_j} + b_{v_j} + b_{w_j} \in \overline x$.
Since each $b_i$ was chosen with a uniform distribution from the coset $a_i$, for each $j$ we have 
\[
    P(b_{u_j} + b_{v_j} + b_{w_j} = x) = \frac{\text{\# of triples of elements in $\overline x$ whose sum is $x$}}{\text{\# of triples of elements in $\overline x$}} = \frac{|Q|^2}{|Q|^3} = 2^{m-n}
\]

Each pair of triples of indices $(u_j, v_j, w_j)$ and $(u_k, v_k, w_k)$ is disjoint if $j \neq k$, and so if we apply the computed probability for the $J$ independent events, we have
\[
    P(b_{u_j} + b_{v_j} + b_{w_j} \neq x, \text{ for all } 1 \leq j \leq J) = (1 - 2^{m-n})^J \leq e^{-J\,2^{m-n}} \leq e^{-\frac{2^{(\sfrac{3}{2})m - n}}{T}},
\]
and due to the initial choice of $m > \frac{2}{3} \log_2(T \ln(2) \, n \, 2^n)$,
\[
    P(b_{u_j} + b_{v_j} + b_{w_j} \neq x, \text{ for all } 1 \leq j \leq J) \leq e^{-\frac{2^{(\sfrac{3}{2})m - n}}{T}} < 2^{-n}.
\]

As the probability of the union of events is at most the sum of the probabilities of each event, we have
\[
    P(\text{$\exists x$ such that $\overline x \not\in A$ and $x$ is not covered by $B$}) \leq (2^n - |Q||A|) \, 2^{-n} < 1,
\]
and hence
\[
    P(\text{$B$ covers each $x$ satisfying $\overline x \not\in A$}) > 0.
\]
Since this probability is positive, there exists a choice of $B$ which covers each $x$ satisfying $\overline x \not\in A$. Let $B_0$ be such a choice.
The set $B_0$ is not necessarily maximal, but we may bound the number of elements required to extend it to a maximal Sidon set. So let $S = B_0 \cup X$ be a maximal Sidon set.
For each element $s_i$ of $X$, $\overline {s_i} \in A$, so let $a_{t_i} = \overline{s_i}$. Then we have $b_{t_i} = s_i + q_i$ for some $q_i \in Q$. But then $q_i = s_i + b_{t_i}$, and since $S$ is Sidon, $q_i \neq q_j$ for $i \neq j$. Thus by the pigeonhole principle, there can be at most $|Q|$ elements in $X$, we have
\[
    |X| \leq |Q| = 2^{n-m} = O(2^{\sfrac{n}{3}}).
\]  
Finally, as $|S| = |B_0| + |X|$ and $|B_0| = |A|$, we achieve 
\[
    |S| \leq |A| + |Q| \leq 2^{m/2} + 2^{n-m} \leq O((n \cdot 2^n)^{\sfrac{1}{3}}) + O(2^{\sfrac{n}{3}}) = O((n \cdot 2^n)^{\sfrac{1}{3}}).
\]
\end{proof}

This result provides an upper bound on the smallest maximal Sidon set, and thus if $S$ is the smallest maximal Sidon set in $\Z_2^n$, we have
\[
    \Omega((2^n)^{\sfrac{1}{3}}) \leq |S| \leq O((n \cdot 2^n)^{\sfrac{1}{3}}).
\]
This pair of bounds parallels the best-known bounds on the minimal size of maximal Sidon sets in the integers $[1, N]$. 

Probabilistic estimates by P. Bennett and T. Bohman \cite{bennett2015note} on the size of greedily-constructed maximal sets in regular hypergraphs predict that in the case of $\Z_2^n$, the right-hand side of this interval is the best-possible upper bound, as a randomly constructed maximal Sidon set in $\Z_2^n$ has size $\Omega((n \cdot 2^n)^{\sfrac{1}{3}})$ with high probability. Thus the minimal size of maximal Sidon sets may be easier to compute here than in the integers, and hence we anticipate that this question will be resolved definitively in the future.

\printbibliography

\end{document}